\documentclass[11pt,draftcls,onecolumn]{IEEEtran}
\usepackage[sub,ovp]{psfragx}
\usepackage{overpic,color}
\usepackage{rotating}
\usepackage{hyperref}
\usepackage[T1]{fontenc}
\usepackage[latin9]{inputenc}
\usepackage{amstext}
\usepackage{graphicx}
\usepackage{amssymb}

\usepackage{amsmath}

\usepackage{amsfonts}
\usepackage{amsthm}
\usepackage{float}
\usepackage{cite}

\usepackage{algorithm}
\usepackage{algpseudocode}

\usepackage[tight,footnotesize]{subfigure}
\usepackage{graphicx}
\usepackage{tikz} 
\usetikzlibrary{backgrounds,shapes}
\usetikzlibrary[arrows,decorations.pathmorphing,backgrounds,positioning,fit,petri]
\usetikzlibrary{snakes}
\usetikzlibrary{calc}

\newtheorem{theorem}{Theorem}[section]

\theoremstyle{definition}

\newcommand{\myvec}[1]{\ensuremath{\mathbf{{#1}}}}   
\newcommand{\mymat}[3]{\ensuremath{\mathbf{#1}^{#2}_{#3}}}        

\newcommand{\mytrace}[1]{\ensuremath{\mathrm{tr\{#1\}}}}

\newcommand{\myref}[1]{(\ref{#1})}
\newcommand{\myhermit}{\ensuremath{H}}

\newcommand{\myrelayvec}{\myvec{x}}

\begin{document}

\title{A Fast Eigen Solution for Homogeneous Quadratic Minimization with at most Three Constraints}
\author{
 \IEEEauthorblockN{Dinesh Dileep Gaurav,~\IEEEmembership{Student Member,~IEEE} and K.V.S. Hari,~\IEEEmembership{Senior Member,~IEEE}}
 \thanks{Dinesh Dileep Gaurav and K.V.S. Hari are with SSP Lab, Department of Electrical Communication Engineering, Indian Institute of Science, Bangalore 560012, India~(e-mail:~\{dileep, hari\}@ece.iisc.ernet.in). The financial support of the DST, India and of the EPSRC, UK under the
 auspices of the India-UK Advanced Technology Centre (IU-ATC) is gratefully
 acknowledged.} 
}
\maketitle
\begin{abstract}
We propose an eigenvalue based technique to solve the Homogeneous Quadratic Constrained Quadratic Programming problem (HQCQP) with at most 3 constraints which arise in many signal processing problems. Semi-Definite Relaxation (SDR) is the only known approach and is computationally intensive. We study the performance of the proposed fast eigen approach through simulations in the context of MIMO relays and show that the solution converges to the solution obtained using the SDR approach with significant reduction in complexity.
\end{abstract}
\begin{IEEEkeywords}
 Homogeneous Quadratic Minimization, Semi-Definite Relaxation, MIMO Relay.
\end{IEEEkeywords}
\section{Introduction}
Fast Algorithms for non-convex Homogeneous Quadratic Constrained Quadratic Programming (HQCQP) are much sought after as many important problems in wireless communications and signal processing can be formulated as HQCQP problems.

In \cite{ieee1634819}, the authors solve a physical layer multi-casting problem where an $M$ antenna transmitter is transmitting the same data  to $K$ single antenna receivers which was formulated as a HQCQP problem with $N=M^2$ variables and $K$ constraints.
In \cite{Chalise_Vanden_2007}, the authors consider the design of optimal relay precoding matrix for an $N$-antenna Multiple-Input Multiple-Output (MIMO) relay which minimizes the relay power while satisfying constraints on Signal-to-Interference-plus-Noise Ratio (SINR) at the $M$ receivers. It is shown that this is a HQCQP with $N^2$ variables and $M$ constraints. Also in \cite{ieee6086630}, a multi-user multi-relay peer-to-peer beamforming problem is considered, which can be reformulated as a HQCQP. In \cite{6193233,5109699}, relay power minimization problems are considered in the context of two-way MIMO relaying which can be reformulated as a HQCQP. In addition, HQCQP problems arise naturally when one applies alternating optimization techniques to several transmit precoder design problems \cite{ieee6295678} and in robust beamforming\cite{Chalise_Luc_2010}.

Semi-Definite Relaxation (SDR) has become a dominant tool to solve HQCQP in recent years\cite{magazine_SDR} but is computationally expensive. In this letter, we address the HQCQP problem with two and three constraints separately. Our main contributions are \textit{(a)} proving that the non-convex HQCQP can be solved by solving an equivalent Eigen Value Problem (EVP) which is convex \textit{(b)} and that the EVP can be solved 
using iterative techniques like search methods leading to significant reduction in computational complexity.
\vspace{-0.2cm}
\section{Problem Statement}
{\let\thefootnote\relax\footnotetext{{\em Notations:} Boldface uppercase and lower case letters denote matrices and vectors respectively. $\mymat{A}{\myhermit}{}$,  {$\mymat{\|A\|}{}{F}$}, $\mytrace{\mymat{A}{}{}}$, $\mymat{A}{-1}{}$ denotes conjugate transpose,
frobenius norm, matrix trace, inverse respectively. $\mymat{A}{}{}\succeq 0$ denotes $\mymat{A}{}{}$ is positive semi-definite. }}
We consider the optimization problem
\cite{magazine_SDR}
\begin{align}
&\min_{\myrelayvec~\epsilon~\mathbb{C}^{N\times 1}}\myvec{x}^{H}\mymat{T}{}{}\myvec{x} \nonumber\\&~s.t.~~\myrelayvec^{H}\mymat{P}{}{i}\myrelayvec+1\leq 0,~\forall i=1,2,3
\label{opt:Formulation1}
\end{align}
where $\mymat{T}{}{}\succ 0$ and all $\mymat{P}{}{i}$ are indefinite Hermitian matrices, and $\myvec{x}$ is the $N\times 1$ complex vector to be found. This quadratic optimization problem is referred to as homogeneous as it does not have a linear term in the complex vector $\myvec{x}$. Using the Cholesky decomposition of $\mymat{T}{}{}$, and letting $\myvec{z}=\mymat{T}{1/2}{}\myvec{x}$ and $\mymat{C}{}{i}=\mymat{T}{-1/2}{}\mymat{P}{}{i}\mymat{T}{-1/2}{}$, we can rewrite \myref{opt:Formulation1} as
\begin{align}
\min_{\myvec{z}~\epsilon~\mathbb{C}^{N\times 1}}&\myvec{z}^{H}\myvec{z},~s.t.~~\myvec{z}^{H}\mymat{C}{}{i}\myvec{z}+1\leq 0,~\forall i=1,2,3
\label{opt:Formulation2}
\end{align}
Note that the objective function is convex but the constraints are not.
\vspace{-0.3cm}
\section{Semi-Definite Relaxation}
The only known tractable approach to \myref{opt:Formulation2} which is globally convergent is Semi-Definite Relaxation[see \cite{magazine_SDR} and reference therein]. To solve \myref{opt:Formulation2}, the semi-definite problem 
\begin{align}
\min_{\mymat{Z}{}{} \in \mathbb{C}^{N^2\times N^2} }~&\mytrace{\mymat{Z}{}{}} \nonumber \\
                        s.t. ~~\mytrace{\mymat{C}{}{i}\mymat{Z}{}{}}+1\leq 0&,~\forall i= 1,2,3 ;~~\mymat{Z}{}{}\succeq 0
\end{align}
is solved. Then one extracts the complex solution vector $\myvec{z}$ from $\mymat{Z}{}{}$ using several techniques as explained in \cite{magazine_SDR}.
\section{Eigen Approach}
In this section, we propose a novel solution for \myref{opt:Formulation2} which can be obtained by solving an eigenvalue problem. For reasons of clarity, we present the two-constraints and the three-constraints cases separately (one-constraint case is trivial).
\vspace{-.3cm}
\subsection{Two-Constraints Case}
Consider the polar decomposition of $\myvec{z}=\sqrt{p}\myvec{u}$ for a positive scalar $p$ and a unit-norm $N\times 1$ vector $\myvec{u}$.
 Rewriting \myref{opt:Formulation2} as
\begin{align}
p^{\star}=\min_{p,\myvec{u}, ||\myvec{u}||_2 = 1}~~p,~s.t.~~
pc_i(\myvec{u})+1\leq 0,~i= 1,2.
\label{opt:scalar_formulation}
\end{align}
where we define $c_i(\myvec{u})=\myvec{u}^{H}\mymat{C}{}{i}\myvec{u}$. Note that \myref{opt:scalar_formulation} is feasible only if there exists some $\myvec{u}\neq \mathbf{0}$ such that $c_i(\myvec{u})<\mathbf{0}$. 
\begin{theorem}
Consider the optimization problem 
\begin{align}
c^{\star}=\min_{\myvec{u},||\myvec{u}||_2 = 1}~\max~\left(c_1(\myvec{u}),c_2(\myvec{u})\right).
\label{opt:minmax}
\end{align}
If $\myvec{u}_{opt}$ is the solution for \myref{opt:minmax}, then it will also be the solution for \myref{opt:scalar_formulation} and $p^{\star}=-1/c^{\star}$.
\label{Theorem:Eqv} 
\end{theorem}
\begin{IEEEproof}
Let $\myvec{u}_p^{\star}$ be the optimum solution for \myref{opt:scalar_formulation} and let $\myvec{u}_c^{\star}$ be the same for \myref{opt:minmax}. Observe that $\myvec{u}_p^{\star}$ will be a feasible solution for \myref{opt:minmax} with objective value\footnote{It is easy to see that at least one of the constraints in \myref{opt:scalar_formulation} should be binding (strictly equal) at the optimum. This is because, if it is not the case, we can decrease $p$ further till one of the constraints becomes equal, thereby decreasing the objective value which is a contradiction.} $-1/p^{\star}$. Similarly, $\myvec{u}_c^{\star}$ will be a feasible solution for \myref{opt:scalar_formulation} with corresponding optimum as $-1/c^{\star}$. Let $p_c^{\star}=-1/c^{\star}$. Thus, we need to prove $p^{\star}=p_c^{\star}$. Assume $p^{\star}<p_c^{\star}$. Thus, $(-1/p^{\star})<c^{\star}$ which implies $\myvec{u}_p^{\star}$ leads to a lower objective value for \myref{opt:minmax} than $\myvec{u}_c^{\star}$, which is a contradiction. Now, for the other direction, assume $p^{\star}>p_c^{\star}$. Clearly, $\myvec{u}_c^{\star}$ leads to a lower objective value for \myref{opt:scalar_formulation} than $\myvec{u}_c^{\star}$ which is a contradiction. Thus $p^{\star}=p_c^{\star}=(-1/c^{\star})$.
\end{IEEEproof}
 It is interesting to note that each $c_i(\myvec{u})$ is a mapping from the unit sphere $\mathbb{U}=\{\myvec{u} \mid \myvec{u}^{H}\myvec{u}=1\}$ to a closed continuous interval on the real line.
The endpoints of this interval are given by the minimum and maximum eigenvalues of $\mymat{C}{}{i}$. The interval is the well known "Numerical Range" of a Hermitian matrix\cite{Gutkin2004143}. Define the two dimensional set
\begin{align}
\mathbb{S}_2(\myvec{u})=\{\left[c_1(\myvec{u}),c_2(\myvec{u})\right] \in \mathbb{R}^2 \mid \myvec{u}^{\myhermit}\myvec{u}=1\}
\end{align}
The set $\mathbb{S}_2(\myvec{u})$ is defined as the joint numerical range of Hermitian matrices $\mymat{C}{}{1}$ and $\mymat{C}{}{2}$\cite{Gutkin2004143}. It follows from the Toeplitz-Hausdorff theorem\cite{Gutkin2004143} that $\mathbb{S}_2(\myvec{u})$ is a closed compact convex subset of the 2-D plane. Note that only for diagonal matrices, this will be a solid rectangle and for other matrices, it will be any closed compact convex shape in general (see appendix \ref{JointNumRange} for an example). Thus, we can interpret \myref{opt:minmax} as minimizing the convex function $\max(x,y)$ over a 2-D convex set $\mathbb{S}_2(\myvec{u})$.
 It can be shown that the solution of  \myref{opt:minmax}  over $\mathbb{S}_2(\myvec{u})$ can occur only at one of the three points in $\mathbb{S}_2(\myvec{u})$, the left-most point on the vertical edge, or the bottom-most point on the horizontal edge, or the extreme bottom-left point which lies on the line $c_1 (\myvec{u})=c_2(\myvec{u})$(see appendix \ref{defPoints}). Let $\myvec{x}_i$ denote the unit eigenvector corresponding to $\lambda_{min}(\mymat{C}{}{i})$. Then, by definition, $c_i(\myvec{x}_i)=\lambda_{min}(\mymat{C}{}{i})$. 
 Thus $(\lambda_{min}(\mymat{C}{}{1}),c_2(\myvec{x}_1))$ will be the left-most point of $\mathbb{S}_2$ and $(c_1(\myvec{x}_2),\lambda_{min}(\mymat{C}{}{2}))$ will be the bottom-most point of $\mathbb{S}_2$. See appendix \ref{maxFunction} for a detailed discussion on the same and the proof on properties of the $\max(x,y)$ function. We state that 
\subsubsection*{\textbf{Case 1}} if $\myvec{x}_1^{\myhermit}\mymat{C}{}{1}\myvec{x}_1>\myvec{x}_1^{\myhermit}\mymat{C}{}{2}\myvec{x}_1$, then $c^{\star}=\lambda_{min}(\mymat{C}{}{1})$ and $\myvec{u}=\myvec{x}_1$. Intuitively, this is the bottom most point of the 2-D set $\mathbb{S}_2(\myvec{u})$.
\subsubsection*{\textbf{Case 2}} if $\myvec{x}_2^{\myhermit}\mymat{C}{}{2}\myvec{x}_2>\myvec{x}_2^{\myhermit}\mymat{C}{}{1}\myvec{x}_2$, , then $c^{\star}=\lambda_{min}(\mymat{C}{}{2})$ and $\myvec{u}=\myvec{x}_2$.  Intuitively, this is the left most point of the  $\mathbb{S}_2(\myvec{u})$.
\subsubsection*{\textbf{Case 3}} if $\myvec{x}_1^{\myhermit}\mymat{C}{}{1}\myvec{x}_1\leq \myvec{x}_1^{\myhermit}\mymat{C}{}{2}\myvec{x}_1$ and $\myvec{x}_2^{\myhermit}\mymat{C}{}{2}\myvec{x}_2\leq \myvec{x}_2^{\myhermit}\mymat{C}{}{1}\myvec{x}_2$
then at the optimal point, $c_1(\myvec{u})=c_2(\myvec{u})$. Defining $\mymat{A}{}{1}=\mymat{C}{}{1}$ and $\mymat{A}{}{2}=\mymat{C}{}{1}-\mymat{C}{}{2}$,
\myref{opt:scalar_formulation} is equivalent to
\begin{align}
c^{\star}=\min_{\myvec{u}, \myvec{u}^{H}\myvec{u}=1}~\myvec{u}^{\myhermit}\mymat{A}{}{1}\myvec{u}, \quad \, s.t. \,\, \myvec{u}^{\myhermit}\mymat{A}{}{2}\myvec{u}=0.
\label{opt:UnitVectorFormulation}
\end{align}
This corresponds to finding the bottom-left point of $\mathbb{S}_2(\myvec{u})$ along the diagonal  $c_1=c_2,~[c_1,c_2]\in\mathbb{S}_2(\myvec{u})$.
 Case $1$ and Case $2$ are equivalent to solving for the minimum eigenvalue of a given Hermitian matrix. In these cases, one of the constraints will not be binding (strict inequality). In Case $3$, we still need to solve an optimization problem. In the following lemma, we prove it is a semi-definite optimization (SDP) problem.
\begin{theorem}
\label{lemmaSDP}
The optimization problem in \myref{opt:UnitVectorFormulation} is equivalent to the following max-min eigenvalue problem.
\begin{align}
c^{\star}=\max_{\lambda,t~\epsilon~\mathbb{R}}~\lambda,~s.t.~\left(\mymat{A}{}{1}+t\mymat{A}{}{2}\right)-\lambda\mymat{I}{}{}\succeq 0
\label{opt:FinalOptProb}
\end{align}
\end{theorem}
\begin{proof} Consider the minimum eigenvalue of $\mymat{A}{}{1}+t\mymat{A}{}{2}$,
\begin{align}
\lambda(t)&=\min \{\myvec{u}^{\myhermit}(\mymat{A}{}{1}+t\mymat{A}{}{2})\myvec{u}\mid \myvec{u}^{\myhermit}\myvec{u}=1\} \label{eq:EigExp}.
\end{align}
By adding one more constraint on the RHS, we have
\begin{align}
\lambda(t)\leq \min\{\myvec{u}^{\myhermit}\mymat{A}{}{1}\myvec{u}\mid \myvec{\myvec{u}}^{\myhermit}\mymat{A}{}{2} \myvec{u}=0, \myvec{u}^{\myhermit}\myvec{u}=1\} \label{eq:BndEig}
\end{align}
Note that the optimization on RHS is same as \myref{opt:UnitVectorFormulation}. We now show that
\begin{align}
\max_{t~\in~\mathbb{R}}~\lambda(t)=\min\{\myvec{u}^{\myhermit}\mymat{A}{}{1}\myvec{u}\mid \myvec{u}^{\myhermit}\mymat{A}{}{2} \myvec{u}=0, \myvec{u}^{\myhermit}\myvec{u}=1\}.
\label{eq:MaxEigExp}
\end{align}
We prove this by contradiction. Suppose that it is not true and there is strict inequality in \myref{eq:MaxEigExp} and that the maximum value of $\lambda(t)$ is achieved at $t^*$. Consider the set
\begin{align}
U=\{\myvec{u}\mid \myvec{u}^{\myhermit}\myvec{u}=1, \lambda(t^*)=\myvec{u}^{\myhermit}\mymat{A}{}{1}\myvec{u}+t^*\myvec{u}^{H}\mymat{A}{}{2}\myvec{u}\}
\label{eq:ProofOpt}
\end{align}
If $\myvec{u}^{\myhermit}\mymat{A}{}{2}\myvec{u}>0$ for all $\myvec{u}\in U$, then there exist $\lambda(t^*+\epsilon)>\lambda(t^*)$ for some $\epsilon>0$, which is a contradiction. Similarly, if $\myvec{u}^{\myhermit}\mymat{A}{}{2}\myvec{u}<0$ for all $\myvec{u}\in U$, then there exist $\lambda(t^*-\epsilon)>\lambda(t^*)$ for some $\epsilon>0$, which is also a contradiction. If there exists $\myvec{u}$ such that $\myvec{u}^{\myhermit}\mymat{A}{}{2}\myvec{u}=0$, then \myref{eq:MaxEigExp} should also hold, which is also a contradiction. Thus $U$ contains some vector $\myvec{v}$ with $\myvec{v}\mymat{A}{}{2}\myvec{v}>0$ and some vector $\myvec{w}$ with $\myvec{w}\mymat{A}{}{2}\myvec{w}<0$. Then some linear combination $\myvec{u}$ of $\myvec{v}$ and $\myvec{w}$ is in $U$ and has $\myvec{u}^{\myhermit}\mymat{A}{}{2}\myvec{u}=0$. Thus, \myref{eq:MaxEigExp} should be true. Thus, we need to maximize the minimum eigenvalue of $\mymat{A}{}{1} + t \mymat{A}{}{2}$ over all $t$. Now, for any Hermitian matrix $\mymat{A}{}{}$, minimum eigenvalue is given by the semi-definite program \cite{magazine_SDR}.
\begin{align}
\lambda_{min}(\mymat{A}{}{})=\max_{\lambda~\epsilon~\mathbb{R}}~\lambda,~\mymat{A}{}{}-\lambda\mymat{I}{}{}\succeq 0.
\label{opt:MaxEigHer}
\end{align}
Using $\mymat{A}{}{}=\mymat{A}{}{1}+t\mymat{A}{}{2}$ in \myref{opt:MaxEigHer} gives the optimization problem in \myref{opt:scalar_formulation} and hence the proof.
\end{proof}

Once the solution is obtained in \myref{opt:FinalOptProb}, we have $\displaystyle p=\frac{1}{\lvert c^{\star}\rvert}$ and $\myvec{u}$ as the eigenvector corresponding to the minimum eigenvalue of $\mymat{A}{}{1}+t^{\star}\mymat{A}{}{2}$ where $t^{\star}$ is the optimal value of $t$ from solving \myref{opt:FinalOptProb}.

 It is worth mentioning that \myref{opt:minmax} has been stated in this form when $\myvec{u}$ is two dimensional and solved using a different approach in \cite{ieee5757641} in the context of Two-way MIMO Relaying. Our approach can be seen as generalizing it for arbitrary dimension.
 \subsection{Three-Constraints Case}
 Now we solve the three-constraints case in a similar manner to the two-constraints case.
 \begin{align}
 \min_{p,\myvec{u}}p,~s.t.~~
 pc_i(\myvec{u})+1\leq 0,~i= 1,2,3
 \label{opt:scalar_formulation3constraint}
 \end{align}
 Following the same arguments as in two-constraints case, solving  \myref{opt:scalar_formulation3constraint} is equivalent to the optimization problem
 \begin{align}
 c^{\star}=\min_{\myvec{u}^H\myvec{u}=1}~\max~\left(c_1(\myvec{u}),c_2(\myvec{u}),c_3(\myvec{u})\right)
 \label{opt:minmax3constraint}.
 \end{align}
 As earlier, we define the set
 \begin{align}
 \mathbb{S}_3(\myvec{u})=\{\left[c_1(\myvec{u}),c_2(\myvec{u}),c_3(\myvec{u})\right] \in \mathbb{R}^3 \mid \myvec{u}^{\myhermit}\myvec{u}=1\}
 \end{align}
 But here, Toeplitz-Hausdorff theorem has to be applied with an exception. $\mathbb{S}_3(\myvec{u})$ is convex for all values of $N\geq3$, but not for $N=2$\cite{Gutkin2004143}. Thus, in this paper, we consider all values of $N\geq 3$. Unlike the two-constraint case, the minimum of the $\max$ function in 3-dimensions is more involved. At the optimum, at least one of the three
 constraints must be strictly binding, which gives us the three following cases.
 \subsubsection*{\textbf{Case 1}} In this case, at the optimum value, all three constraints are binding. Thus, the solution will be lying on the line $c_1(\myvec{u})=c_2(\myvec{u})=c_3(\myvec{u})$. Note that, there is only one such point. Defining $\mymat{A}{}{1}=\mymat{C}{}{1}$,  $\mymat{A}{}{2}=\mymat{C}{}{1}-\mymat{C}{}{2}$ and $\mymat{A}{}{3}=\mymat{C}{}{1}-\mymat{C}{}{3}$, this is equivalent to solving the optimization problem,
 \begin{align}
c^{\star}=&\min_{\myvec{u}^{H}\myvec{u}=1}~\myvec{u}^{\myhermit}\mymat{A}{}{1}\myvec{u}\nonumber\\ s.t.~&\myvec{u}^{\myhermit}\mymat{A}{}{2}\myvec{u}=0, \quad
\myvec{u}^{\myhermit}\mymat{A}{}{3}\myvec{u}=0
\label{opt:UnitVectorFormulationCase1}.
 \end{align}
  Using the same arguments as in the proof for theorem \ref{lemmaSDP}, we can prove \myref{opt:UnitVectorFormulationCase1} is equivalent to the semi-definite optimization problem
  \begin{align}
  c^{\star}=\max_{\lambda,t_1,t_2~\epsilon~\mathbb{R}}~\lambda,~~\left(\mymat{A}{}{1}+t_1\mymat{A}{}{2}+t_2\mymat{A}{}{3}\right)-\lambda\mymat{I}{}{}\succeq 0
  \label{opt:Case1Constraint3}
  \end{align}
   \subsubsection*{\textbf{Case 2}} In this case, any two of the constraints will be binding at the optimum. Note that there are three such points in $\mathbb{S}_3(\myvec{u})$. Without loss of generality, assume the first and second constraint will be binding at the optimum. Thus, we have $c_1(\myvec{u})=c_2(\myvec{u})$ at the optimum. Then, defining $\mymat{A}{}{1}=\mymat{C}{}{1}$ and $\mymat{A}{}{2}=\mymat{C}{}{1}-\mymat{C}{}{2}$, we have
   \begin{align}
   c^{\star}=\min_{\myvec{u}^{H}\myvec{u}=1}~\myvec{u}^{\myhermit}\mymat{A}{}{1}\myvec{u},\quad \myvec{u}^{\myhermit}\mymat{A}{}{2}\myvec{u}=0
   \label{opt:Case2Constraint3}
   \end{align}
   Note that we converted the same problem into a SDP problem in lemma \ref{lemmaSDP}. Thus we can apply those results here to rewrite it as in \myref{opt:FinalOptProb}.
   \subsubsection*{\textbf{Case 3}}In this case, only one of the constraints will be binding at optimum. Note that there are three such possibilites. It can be shown that if $i^{th}$ constraint is binding, then $c^{\star}=\lambda_{min}(\mymat{C}{}{i})$.

   In the three-constraints case, we are required to find all the seven 3-dimensional points in $\mathbb{S}_3 (\myvec{u})$, corresponding to each of this above mentioned cases. After evaluating the cost function at these seven points, the minimum value yields $c^{\star}$.
   In numerical examples, we show through extensive simulations that this is still significantly faster than SDR even though it has to solve 4 different optimization problems corresponding to case 1 and case 2.
 \section{Iterative Algorithms}
 In this section, we present iterative algorithms to solve two semi-definite optimization problems in \myref{opt:FinalOptProb} and in \myref{opt:Case1Constraint3}. Note that, they could also be solved with a convex package. Let $\mymat{M}{}{0}$, $\mymat{M}{}{1}$  and $\mymat{M}{}{2}$ be  $N\times N$ Hermitian matrices. We consider the optimization problems
 \begin{align}
 \max_{\lambda,t_1~\epsilon~\mathbb{R}}~\lambda,~s.t.~\left(\mymat{M}{}{0}+t_1\mymat{M}{}{1}\right)-\lambda\mymat{I}{}{}\succeq 0
 \label{iterativeproblem1}
 \end{align}
and
\begin{align}
 \max_{\lambda,t_1,t_2~\epsilon~\mathbb{R}}~\lambda,~s.t.~\left(\mymat{M}{}{0}+t_1\mymat{M}{}{1}+t_2\mymat{M}{}{2}\right)-\lambda\mymat{I}{}{}\succeq 0.
 \label{iterativeproblem2}
 \end{align}
 Optimization problem in \myref{iterativeproblem1} is required to solve the  case 3 in two-constraints case, and Case $2$ in three-constraints case. Similarly, the one in \myref{iterativeproblem2} is needed to solve the Case 1 in three-constraints case. We observe that both the problems are concave in the variables. To be specific, we are maximizing the following concave  functions
 \begin{align}
\lambda(t_1)=\lambda_{min}\left(\mymat{M}{}{0}+t_1\mymat{M}{}{1}\right),~t_1\in \mathbb{R}
 \end{align}
 and
 \begin{align}
 \lambda(t_1,t_2)=\lambda_{min}\left(\mymat{M}{}{0}+t_1\mymat{M}{}{1}+t_2\mymat{M}{}{2}\right),~(t_1,t_2)\in \mathbb{R}^2,
 \end{align}
 respectively. Note that these come under the well known problem of maximizing the minimum eigenvalue of a linear combination of Hermitian matrices which is known to be a convex optimization problem\cite{Vandenberghe96semidefiniteprogramming}. Thus, to solve \myref{iterativeproblem1}, any standard one dimensional-search can be used.
 Similarly, to solve \myref{iterativeproblem2}, we propose an alternating optimization approach by 
 fixing one variable and optimizing over the other. Since this is a convex optimization problem, 
 it is globally convergent. It is worth mentioning that we could resort to advanced optimization techniques 
 which might improve the speed of convergence rather than simple search methods. 
 But, search methods are attractive since they are efficient to implement in practice. 
 We implement a standard dichotomous search \cite{Antoniou_Lu_2007} with the initial search 
 interval as $[\lambda_{min}(\mymat{M}{}{0}),0]$ in both the cases. 
 We scaled the interval size whenever $\lvert \lambda_{min}(\mymat{M}{}{0})\rvert$ is very large,
  to accelerate the convergence. It is interesting to note that the majority of computation in each iteration of both problems comes from the minimum eigenvalue calculation.
\section{Simulations}
\begin{figure}[ht!]
    \centering
         \includegraphics[height=6cm,width=8.5cm]{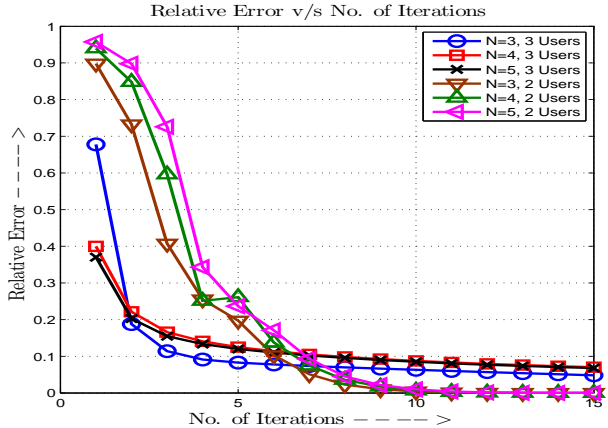}          
        \label{fig:2n3constraintConverg}
\centering      \caption{ MIMO Relay scenario: Relative Error, $\epsilon_i$ as a function of number of iterations for 2-users and 3-users cases with varying number of antennas $(M)$ at the relay.} 
\end{figure}
For the purpose of simulation, we consider a practical application to compare SDR and the proposed technique. We consider the MIMO Relay Precoder design problem solved in \cite{Chalise_Vanden_2007} which minimizes the total relay power subject to SINR constraints at the receivers which is shown to be a HQCQP with number of constraints equal to the number of users, and number of variables is equal to the square of number of antennas. Thus, the proposed technique can be used in the 2-user and the 3-user scenarios. For both the scenarios, we randomly generate $1000$ standard complex Gaussian MIMO channels corresponding to $M\in\{3,4,5\} $ where $M$ denotes the number of antennas at the Relay. Thus, there are $N=M^2$ variables in the correspoding HQCQP. The noise variances at all nodes are fixed to be $-10~$dB and SINR requirement to be $3~$dB in all scenarios. Interested reader is urged to refer \cite{Chalise_Vanden_2007} for a detailed system model and thorough theoretical analysis.

Numerical simulations are used to compare two different aspects of the proposed technique when compared against SDR. For solving SDR, we used CVX, a package for solving convex programs\cite{cvx}. In the first set of simulations, we look at the \emph{Average Relative Error} of the proposed technique in approaching the true objective value  at the optimum which will be obtained from SDR. For the proposed technique, we define the Relative Error $\epsilon_i(n)$ in the $i^{th}$ iteration for the $n^{th}$ problem data as
\begin{align}
\epsilon_i(n)=\lvert\frac{p^{\star}(n)-\hat{p_i}(n)}{p^{\star}(n)}\rvert
\end{align}
where $p^{\star}(n)$ denotes the true value obtained from SDR and $\hat{p_i}(n)$  is the value obtained from the current iteration of the proposed technique for the $n^{th}$ data set. The \textit{Average Relative Error} for a particular iteration is then calculated over 1000 such datasets by averaging the Relative Error. Average Relative Error gives an idea of how fast the algorithm converges to the true value in terms of number of iterations.  

Fig.\ref{fig:2n3constraintConverg} shows the relative error performance of the proposed technique for the 2-constraints (2 users) case and 3-constraints (3-users) case.  It is observed that for various $N$ being considered, the proposed technique reaches $90\%$  of the true value in 8-10 iterations.

In another set of simulations, we compare the raw CPU times of both SDR and the proposed technique on a standard Desktop PC configuration. The stopping criterion for the dichotomous search in the proposed algorithm is based on a threshold value of $10^{-4}$ for the final interval size, while the SDR algorithm uses an internal threshold value for convergence \cite{cvx}.  \tablename{~\ref{Tab:Speed}} presents a  comparison of CPU times for different scenarios. Each entry inside the table specifies the ratio of CPU time based on SDR algorithm and the CPU time  based on the proposed algorithm. For instance, for the 2-users case and having 3 antennas, the proposed technique is $720$ times faster than SDR based solution. We can see that as the number of antennas increases, the ratio decreases. This can be considered to be due to the limitation of the dichotomous algorithm and also the increase in computational cost of evaluating eigenvalues of a larger matrix. For the 3-users case, the ratio is significantly reduced due to the nature of the cost function and solving multiple eigenvalue problems (compared to the 2-users case). However, the improvement is still significant. Overall, there is significant improvement in the reduction of complexity of the HQCQP minimization problem.
\begin{table}
\caption{MIMO Relay scenario: Ratio of CPU time required for SDR based method to proposed Eigen based method for 2-users and 3-users case with varying number of antennas, $M$.}
\begin{center}
    \begin{tabular}{| l | l | l | l |}
    \hline
      & $M=3$& $M=4$ & $M=5$ \\ \hline
    2 Users & 720 & 304 & 127 \\ \hline
    3 Users & 54 & 26 & 16\\ \hline

    \hline
    \end{tabular}
\end{center}
\vspace{-0.7cm}
\label{Tab:Speed}
\end{table} 
\vspace{-0.3cm}
\section{Conclusion}
In this paper, we propose an eigen-framework to solve HQCQP with at most three constraints. We propose iterative techniques which have significant reduction in complexity to solve the same. We demonstrate the benefits using a MIMO relay power minimization problem.
\begin{appendices}
\section{An Introduction to Joint Numerical Range of Hermitian Matrices}
\label{JointNumRange}
Let $\mymat{C}{}{1},\mymat{C}{}{2},\dots,\mymat{C}{}{m}$ be $N \times N$ Hermitian matrices. Define $c_i(\myvec{u})=\myvec{u}^H\mymat{C}{}{i}\myvec{u}$ for any unit-norm vector $\myvec{u}$ for a given $i\in\{1,\dots,m\}$. Note that $c_i(\myvec{u})$ is always real. Then the set of points in $\mathbb{R}^m$ defined as  
\begin{align}
\mathbb{S}_m=\{~\left[c_1(\myvec{u}),c_2(\myvec{u}),\dots,c_m(\myvec{u})\right]\in \mathbb{R}^{m},\myvec{u}^H\myvec{u}=1\}
\end{align}
is defined as the joint numerical range of Hermitian matrices $\mymat{C}{}{1},\mymat{C}{}{2},\dots,\mymat{C}{}{m}$. When $m=1$, this reduces to the well-known Rayleigh ratio of a Hermitian matrix which varies from $\lambda_{min}(\mymat{C}{}{1})\leq c_i(\myvec{u})\leq \lambda_{max}(\mymat{C}{}{1})$ continuously over the unit sphere $\myvec{u}^H\myvec{u}=1$\cite{Horn_Johnson_1991}. When $m=2$, it becomes a $2-D$ set. From the well known Toeplitz-Hausdorff Theorem \cite{Gutkin2004143}, it follows that $\mathbb{S}_2$ is a closed compact convex set. We provide some pictorial examples for the same. Let $N=4$ and $m=2$. We consider two randomly generated hermitian matrices $\mymat{C}{}{1},\mymat{C}{}{2}$ and generate the corresponding $\mathbb{S}_2$.
\begin{align}
\mymat{C}{}{1} =  \begin{bmatrix} 0.6487        &     0.4814 - 0.9681i &  -0.9725 - 1.0067i & -0.5501 - 0.8097i \\
   0.4814 + 0.9681i &  -0.2919  &           0.0625 + 1.0729i &  -0.4020 + 0.2913i \\
  -0.9725 + 1.0067i &  0.0625 - 1.0729i & -2.1321    &         0.4457 - 0.4558i \\
  -0.5501 + 0.8097i &  -0.4020 - 0.2913i &  0.4457 + 0.4558i & -1.4286    \end{bmatrix}      \\
\mymat{C}{}{2}=  \begin{bmatrix}
    0.8810         &    1.0909 - 0.2062i & -0.3353 + 0.3102i & -0.1632 - 0.8746i \\
     1.0909 + 0.2062i &  -0.6045   &         -0.4007 - 0.0412i & -0.1488 + 0.5651i \\
    -0.3353 - 0.3102i &  -0.4007 + 0.0412i &  -0.4677   &          0.3299 + 1.0372i \\
    -0.1632 + 0.8746i & -0.1488 - 0.5651i &  0.3299 - 1.0372i  & 0.3086           
  \end{bmatrix}
\end{align}
The joint numerical range for the above $\mymat{C}{}{1},\mymat{C}{}{2}$ is given in Fig:\ref{Fig:NumRange}
\begin{figure}[ht!]
\centering
\begin{overpic}[height=.4\textheight,width=.7\textwidth]{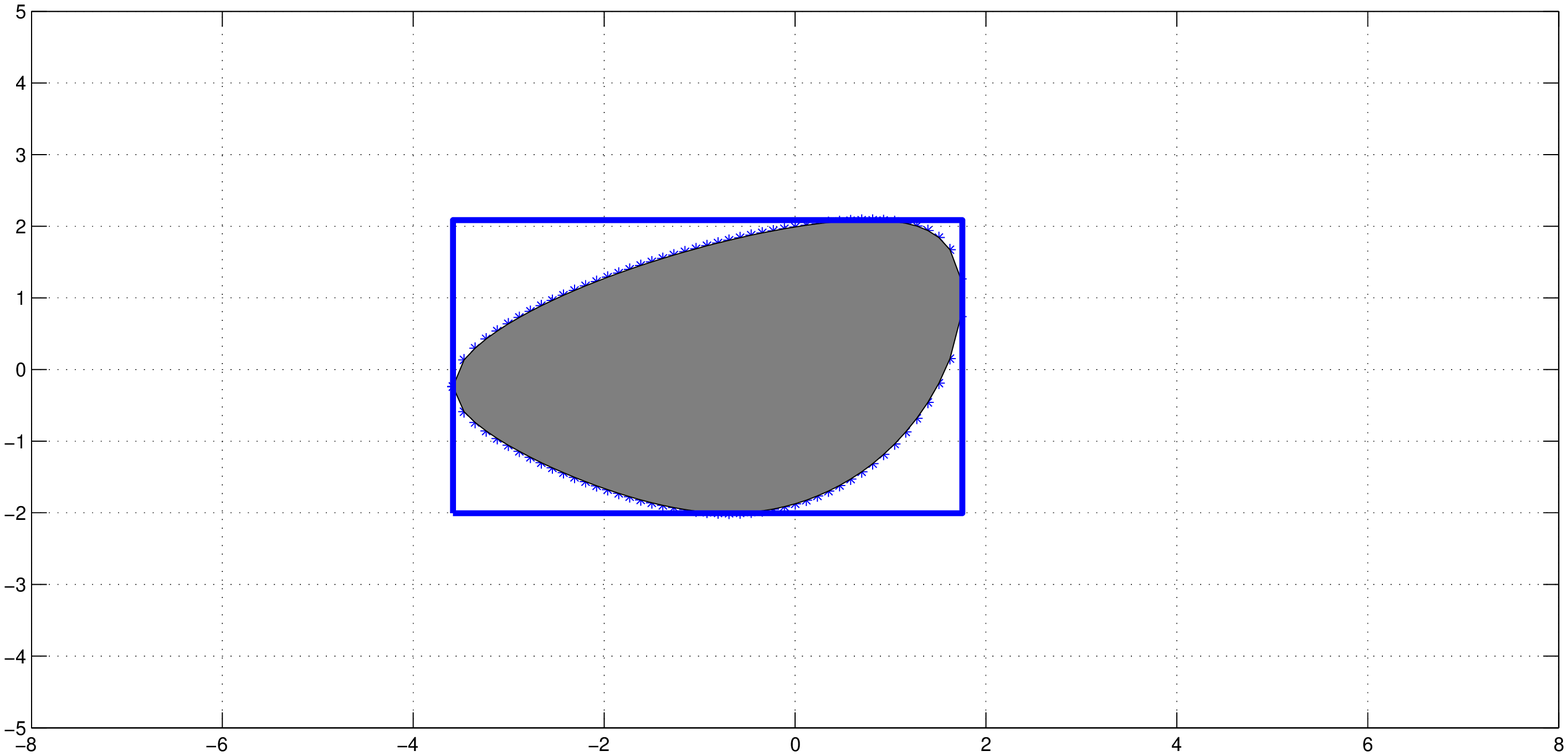}
\put(20,24){\tiny $\left(\lambda_{min}(\mymat{C}{}{1}),\lambda_{min}(\mymat{C}{}{2})\right)$}
\put(62,53){\tiny $\left(\lambda_{max}(\mymat{C}{}{1}),\lambda_{max}(\mymat{C}{}{2})\right)$}

\put(45,40){\huge \bf {$\mathbb{S}_2$}}

\put(25,0){\scriptsize <----- $\lambda_{min}(\mymat{C}{}{1})\leq \myvec{u}^H\mymat{C}{}{1}\myvec{u}\leq\lambda_{max}(\mymat{C}{}{1}),~{\lvert\lvert\myvec{u}\rvert\rvert_2=1}$ ----->}
\put(8,10){\scriptsize \begin{sideways}\tiny <----- $\lambda_{min}(\mymat{C}{}{2})\leq \myvec{u}^H\mymat{C}{}{2}\myvec{u}\leq\lambda_{max}(\mymat{C}{}{2}),~{\lvert\lvert\myvec{u}\rvert\rvert_2=1}$ ----->\end{sideways}}
\put(14,75){\small Plotting 2-D set  $\mathbb{S}_2=\{~\left[\myvec{u}^H\mymat{C}{}{1}\myvec{u},\myvec{u}^H\mymat{C}{}{2}\myvec{u}\right]\in \mathbb{R}^2,\myvec{u}^H\myvec{u}=1\}$
}
\put(33.2,26.3){\tikz\draw[fill=red] (0,0)
ellipse (3pt and 3pt);}
\put(59.4,52){\tikz\draw[fill=red] (0,0)
ellipse (3pt and 3pt);}
\end{overpic}
\label{Fig:NumRange}
\end{figure}
The greyed solid region in Fig:\ref{Fig:NumRange} shows the 2-D joint numerical range of given hermitian matrices $\mymat{C}{}{1}$ and $\mymat{C}{}{2}$. This clearly shows that it is not rectangular. It is straightforward to see that the Joint Numerical Range in 2-D will be a solid rectangle if the matrices are diagonal. 
\section{Some Definitions}
\label{defPoints}
Consider any $2-D$ closed compact and convex set $\mathbb{S}_2$. We define the following points in this set
\begin{itemize}
\item {\textbf{Left-Most Point:}} We define the left-most point of set $\mathbb{S}_2$ as the point whose co-ordinates $(x_{w},y_{w})$ are given by 
\begin{align}
x_l=\min_{(x,y)\in \mathbb{S}_2}~x,~~~y_l=\min_{(x_l,y)\in \mathbb{S}_2}~y
\end{align}
Intuitively, This is the point which lies on the bottom on the left most vertical edge. Note that the left-most vertical edge can be a single point or multiple points.
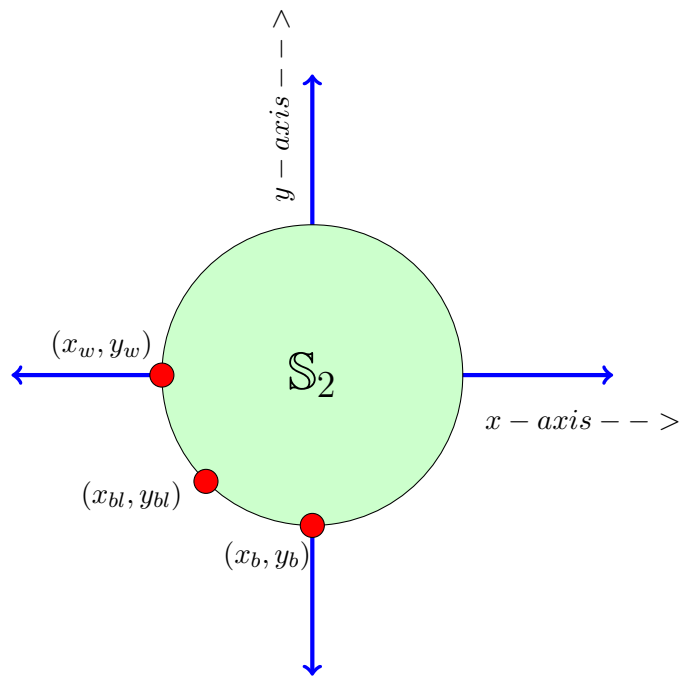
\begin{figure}[ht!]
    \centering
\begin{tikzpicture}[scale=2]
\draw[blue,ultra thick,<->] (-2,0) -- (2,0);
\draw[blue,ultra thick,<->] (0,-2) -- (0,2);
\draw[] (0,0)[fill=green!20]  circle (1cm);
\draw (-1,0) [fill=red] circle (.08cm);
\draw (0,-1)  [fill=red] circle (.08cm);
\draw (-0.707,-0.707) [fill=red] circle (.08cm);
\node at (0,0) {\huge \bf $\mathbb{S}_2$};
\node at (-0.3,-1.2) {\bf $(x_{b},y_{b})$};
\node at (-1.2,-.8) {$(x_{bl},y_{bl})$};
\node at (-1.4,.2) {$(x_{w},y_{w})$};
\node at (-.2,1.8) {\begin{sideways}$y-axis -->$\end{sideways}};
\node at (1.8,-.3) {$x-axis -->$};
\end{tikzpicture}
        \label{fig:2dexample}
        \centering\caption{~~~Demonstration of Defined Points Left-Most, Bottom-Most and Bottom-Left Most points}        
\end{figure}
\item {\textbf{Bottom-Most Point:}} We define the bottom-most point of set $\mathbb{S}_2$ as the point whose co-ordinates $(x_{s},y_{s})$ are given by 
\begin{align}
y_{b}=\min_{(x,y)\in \mathbb{S}_2}~y,~~~x_{b}=\min_{(x,y_{b})\in \mathbb{S}_2}~x
\end{align}
Intuitively, This is the point which lies on the left-most end on the bottom-most horizontal edge.
\item \textbf{Bottom-Left-Most Point:}  We define the bottom-left-most point of set $\mathbb{S}_2$ as the point whose co-ordinates $(x_{bl},y_{bl})$ are given by 
\begin{align}
x_{bl}=\min_{(x,y)\in \mathbb{S}_2,x=y}~x ,~~~y_{bl}=x_{bl}
\end{align}
Intuitively, this point lies on the lowest end of the line $x=y$ towards the bottom left side
\end{itemize}. In Fig:\ref{fig:2dexample}, we give an example of all the three points for a given convex set which in this case is a circular disk. 
\section{Bottom-Most, Left-Most and Bottom-Left Most point for $\mathbb{S}_2$}
Let $\myvec{x}_i$ denote the eigenvector corresponding to $\lambda_{min}(\mymat{C}{}{i})$. Thus by definition, $$c_i(\myvec{x_i})=\myvec{x}_i^H\mymat{C}{}{i}\myvec{x}=\lambda_{min}(\mymat{C}{}{i})$$. In $\mathbb{S}_2$, the $x-axis$ corresponds to the numerical range of $\mymat{C}{}{1}$ which ranges from $\lambda_{min}(\mymat{C}{}{1})\leq c_1(\myvec{u})\leq \lambda_{max}(\mymat{C}{}{1})$. Similarly, the $y-axis$ range will be $\lambda_{min}(\mymat{C}{}{2})\leq c_2(\myvec{u})\leq \lambda_{max}(\mymat{C}{}{2})$. Thus, the bottom-most would correspond to  $(c_1(\myvec{x}_2),\lambda_{min}(\mymat{C}{}{2}))$ where  $c_1(\myvec{x}_2)=\myvec{x}_2^H\mymat{C}{}{2}\myvec{x}_2$. The Bottom-Left point will correspond to extreme bottom point lying on line $x=y$ i.e. $c_1(\myvec{u})=c_2(\myvec{u})$. It can be obtained by solving the problem
\begin{align}
\min_{||\myvec{u}||_2=1} c_1(\myvec{u}),~~c_1(\myvec{u})=c_2(\myvec{u})
\end{align}
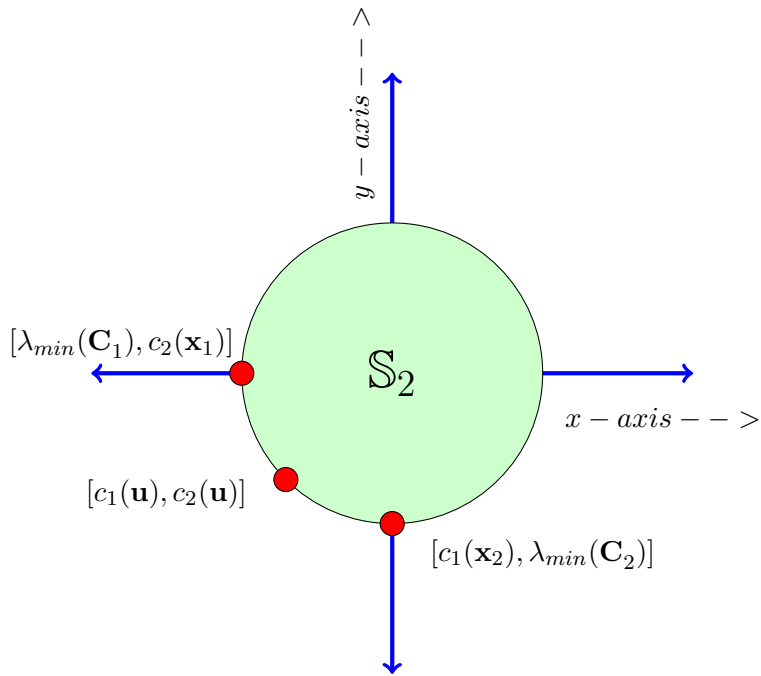
\begin{figure}[ht!]
    \centering
\begin{tikzpicture}[scale=2]
\draw[blue,ultra thick,<->] (-2,0) -- (2,0);
\draw[blue,ultra thick,<->] (0,-2) -- (0,2);
\draw[] (0,0)[fill=green!20]  circle (1cm);
\draw (-1,0) [fill=red] circle (.08cm);
\draw (0,-1)  [fill=red] circle (.08cm);
\draw (-0.707,-0.707) [fill=red] circle (.08cm);
\node at (0,0) {\huge \bf $\mathbb{S}_2$};
\node at (1,-1.2) {\bf $\left[c_1(\myvec{x}_2),\lambda_{min}(\mymat{C}{}{2})\right]$};
\node at (-1.5,-.8) {$\left[c_1(\myvec{u}),c_2(\myvec{u})\right]$};
\node at (-1.8,.2) {$\left[\lambda_{min}(\mymat{C}{}{1}),c_2(\myvec{x}_1)\right]$};
\node at (-.2,1.8) {\begin{sideways}$y-axis -->$\end{sideways}};
\node at (1.8,-.3) {$x-axis -->$};
\end{tikzpicture}
        \label{fig:convexshapeeg}
        \centering\caption{~~~Demonstration of Defined Points Left-Most, Bottom-Most and Bottom-Left Most points for the Joint Numerical Range}  
\end{figure}
We consider a circular disk shape in Fig:\ref{fig:convexshapeeg} to explain this points. We have marked the Bottom-Most, Left-Most and the Bottom-Left Most point in the given shape corresponding to the joint numerical range of the given two hermitian matrices. 
\section{$\max$ Function over a convex 2-D set}
\label{maxFunction}
We define the $\max$ function as 
\begin{align} \max(x,y)= \begin{cases} x, &\mbox{if } x \geq y \\
y, & \mbox{if } x < y. \end{cases} \end{align}
Thus, at a point $(x,y)$, calculating the function $\max(x,y)$ amounts to calculating the maximum value among its two arguments. It is easy to see that $\max$ is indeed a convex function from the definition of convex functions itself. We consider the problem of finding the global optimum $(x^{\star},y^{\star})\in \mathbb{S}_2$ of the problem 
\begin{align}
c^{\star}=\min_{(x,y)\in\mathbb{S}_2}\max{(x,y)} 
\label{opt:prob1}
\end{align}  
\begin{theorem}
Consider the $\max$ function over a given $2D$ closed, compact, convex set $\mathbb{S}_2$. Let us assume the points $(x_{b},y_{b})$ Bottom-Most Point, $(x_l,y_l)$ Left-Most Point and $(x_{bl},y_{bl})$ Bottom-Left-Most Point are given for the set $\mathbb{S}_2$. Then  
\begin{enumerate}
\item The global optimum of the minimum of $\max$ function over the convex set $\mathbb{S}_2$ will be achieved at one of the points among the Left-Most $(x_l,y_l)$ , Bottom-Most $(x_{b},y_{b})$, Bottom-Left most point $(x_{bl},y_{bl})$ of the convex set $\mathbb{S}_2$. 
\item Moreover, we have the alternative
\begin{enumerate}[]
\item \label{proof:Part2A} Either if $x_l>y_l$, then  $c^{\star}=x_l$
\item or if $y_{b}>x_{b}$, then  $c^{\star}=y_{b}$
\item or if $y_l >x_l$ and $y_{b}<x_{b}$, then $c^{\star}=x_{bl}$
\end{enumerate}  
\end{enumerate}
\end{theorem}

\begin{proof}
\begin{enumerate}
\item   Suppose $(a,b)$ is a global minimizer of $\max(x,y)$. If $a=b$, clearly $(a,b)$ must be the bottom left endpoint of the slice of $\mathbb{S}$ lying on the line $x=y$ which is the Bottom-Left Most point. Thus $c^{\star}=x_{bl}$ in this case.

Now suppose $a\not=b$. WLOG, assume that $a<b$ (hence $\max(a,b)=b$). Consider the Bottom-Most point $(x_{b},y_{b})$ of the set $\mathbb{S}_2$. We consider the cases
\begin{itemize}
\item $x_{b}<a$. Since $y_{b}\le b$ and $\max(a,b)\le \max(x_{b},y_{b})$, we must have $y_{b}=b$ and hence $\max(a,b)=\max(x_{b},y_{b})$.
\item $x_{b}\ge a$. Consider the convex combination $(x,y)=(\theta x_{b}+(1-\theta)a,\ \theta y_{b}+(1-\theta)b)$. This denotes all the points lying in the line segment between $(a,b)$ and $(x_{b},y_{b})$.  Since we have assumed that $a<b$, we get $\max(x,y) = \theta y_{b}+(1-\theta)b$ when $\theta$ is small. In order that this is greater than or equal to $b=\max(a,b)$, we must have $y_{b}\ge b$. However, by definition of $(x_{b},y_{b})$, this means $(a,b)=(x_{b},y_{b})$.
\end{itemize}
In a similar line, it is easy to prove that if $a>b$, then $(x_l,y_l)$ is the solution. 
\item We prove \ref{proof:Part2A}. Assume $x_l>y_l$. From the definition of $(x_l,y_l)$ and $(x_{b},y_{b})$, it follows that $x_l \leq x_{b}$ and $y_{b} \leq y_l$. Then $x_{b} > y_{b}$ (because $y_{b} \leq y_l < x_l \leq x_{b}$). Thus $\max(x_l,y_l)=x_l\leq x_{b} = \max(x_{b},y_{b})$ and hence $(x_l,y_l)$ should be the optimum.  
\end{enumerate}
\end{proof}
\end{appendices}
\bibliographystyle{IEEEtran}
\bibliography{IEEEabrv,bibdata}
\end{document}